\documentclass{amsart} 

\usepackage[english]{babel}
\usepackage[latin1]{inputenc}

\usepackage{amsmath}
\usepackage{amsthm}
\usepackage{amssymb}
\usepackage{tikz}
\usepackage{csquotes}

\usepackage{imakeidx} 
\usepackage{appendix}
\usepackage{tikz-cd}
\usepackage{verbatim}
\usepackage{enumitem}
\usepackage{multicol}
\usepackage{adjustbox}
\usepackage{mathtools}

\usepackage{aliascnt}
\usepackage[backref=page]{hyperref}
\usepackage{cleveref}

\hypersetup{colorlinks=true,linkcolor=blue,anchorcolor=blue,citecolor=blue}

\newtheorem{thm}{Theorem}[section]
\crefname{thm}{Theorem}{Theorems}
\Crefname{thm}{Theorem}{Theorems}

\newaliascnt{prop}{thm}
\newtheorem{prop}[prop]{Proposition}
\aliascntresetthe{prop}
\crefname{prop}{Proposition}{Propositions}
\Crefname{prop}{Proposition}{Propositions}

\newaliascnt{lemma}{thm}
\newtheorem{lemma}[lemma]{Lemma}
\aliascntresetthe{lemma}
\crefname{lemma}{Lemma}{Lemmas}
\Crefname{lemma}{Lemma}{Lemmas}

\newaliascnt{cor}{thm}

\aliascntresetthe{cor}
\crefname{cor}{Corollary}{Corollaries}
\Crefname{cor}{Corollary}{Corollaries}

\newaliascnt{conj}{thm}

\aliascntresetthe{conj}
\crefname{conj}{Conjecture}{Conjectures}
\Crefname{conj}{Conjecture}{Conjectures}

\newaliascnt{question}{thm}

\aliascntresetthe{question}
\crefname{question}{Question}{Questions}
\Crefname{question}{Question}{Questions}

\theoremstyle{definition}

\newaliascnt{defin}{thm}

\aliascntresetthe{defin}
\crefname{defin}{Definition}{Definitions}
\Crefname{defin}{Definition}{Definitions}

\newaliascnt{construction}{thm}

\aliascntresetthe{construction}
\crefname{construction}{Construction}{Constructions}
\Crefname{construction}{Construction}{Constructions}

\newaliascnt{example}{thm}

\aliascntresetthe{example}
\crefname{example}{Example}{Examples}
\Crefname{example}{Example}{Examples}

\newaliascnt{claim}{thm}

\aliascntresetthe{claim}
\crefname{claim}{Claim}{Claims}
\Crefname{claim}{Claim}{Claims}

\newaliascnt{exercise}{thm}

\aliascntresetthe{exercise}
\crefname{exercise}{Exercise}{Exercises}
\Crefname{exercise}{Exercise}{Exercises}

\newaliascnt{assumption}{thm}

\aliascntresetthe{assumption}
\crefname{assumption}{Assumption}{Assumptions}
\Crefname{assumption}{Assumption}{Assumptions}

\theoremstyle{remark}

\newaliascnt{rmk}{thm}
\newtheorem{rmk}[rmk]{Remark}
\aliascntresetthe{rmk}
\crefname{rmk}{Remark}{Remarks}
\Crefname{rmk}{Remark}{Remarks}

\numberwithin{equation}{section} 

\newcommand{\Z}{\mathbb Z}

\renewcommand{\P}{\mathbb P}

\newcommand{\mc}[1]{\mathcal{#1}}
\newcommand{\cl}{\overline}
\newcommand{\set}[1]{\left\{#1\right\}}
\renewcommand{\phi}{\varphi}

\newcommand{\on}[1]{\operatorname{#1}}

\title{Equivariant intermediate Jacobians and intersections of two quadrics}

\address{CNRS\\
	Institut Galil\'ee\\
	Universit\'e Sorbonne Paris Nord\\    
	93430, Villetaneuse, France}

\author{Federico Scavia}
\email{scavia@math.univ-paris13.fr}

\makeatletter
\@namedef{subjclassname@2020}{\textup{2020} Mathematics Subject Classification}
\makeatother

\date{May 5, 2026}


\begin{document}

\begin{abstract}
    We present a short proof of the following theorem of Hassett and Tschinkel: for every finite group $G$, a $G$-equivariant smooth complete intersection of two quadrics in $\mathbb{P}^5_{\mathbb{C}}$ is projectively $G$-linear if and only if it contains a $G$-invariant line.
\end{abstract}

\maketitle
	
	\section{Introduction}
	
	Let $k$ be an algebraically closed field, let $G$ be a finite group, and let $X$ be a $G$-variety, that is, a separated integral $k$-scheme of finite type endowed with a $G$-action over $k$. We say that $X$ is \emph{projectively $G$-linear} if there exist a group homomorphism $\rho\colon G\to \on{PGL}_{n+1}$ and a $G$-equivariant birational equivalence between $X$ and $\P^n_k$, where $G$ acts on $\P^n_k$ via $\rho$. 
    
    In this note, we give an alternative proof of the following theorem of Hassett and Tschinkel \cite[Theorem 24]{hassett2022equivariant}.
    
	\begin{thm}[Hassett--Tschinkel]\label{hassett-tschinkel}
		Let $k$ be an algebraically closed field of characteristic zero, let $G$ be a finite subgroup of $\on{PGL}_6$, and let $X\subset \P^5_k$ be a smooth complete intersection of two quadrics such that the restriction of the natural $\on{PGL}_6$-action on $\P^5_k$ to $G$ leaves $X$ invariant. Then $X$ is projectively $G$-linear if and only if it contains a $G$-invariant line.	
	\end{thm}
    \Cref{hassett-tschinkel} is an equivariant version of the following result: for a field $k$, a smooth complete intersection of two quadrics in $\mathbb{P}^5_k$ is rational if and only if it contains a line. This is due to Hassett and Tschinkel \cite{hassett2021rationality} when $k=\mathbb{R}$, and to Benoist and Wittenberg \cite[Theorem A]{benoist2023intermediate} for arbitrary $k$. The original proof of \Cref{hassett-tschinkel} reduces to this result by a twisting argument; see \cite[\S 8.2]{hassett2022equivariant}. In this note, we present a more direct proof of \Cref{hassett-tschinkel}, which instead proceeds by adapting the arguments of \cite{benoist2023intermediate} to the $G$-equivariant context. 
    
    The structure of our proof of \Cref{hassett-tschinkel} is as follows. If $X$ contains a $G$-invariant line $L$, then by projection from $L$ we see that $X$ is projectively $G$-linear. For the converse, we first derive in \Cref{bw-3.10} a necessary condition for the projective $G$-linearity of a rational smooth projective $G$-variety of dimension $3$ by considering $G$-equivariant torsors under its intermediate Jacobian (this is a $G$-equivariant analogue of \cite[Proposition 3.10]{benoist2023intermediate}). For a smooth complete intersection $X$ of two quadrics in $\mathbb{P}^5_k$, we make this condition explicit in terms of the Fano variety $F$ of lines in $X$, which is a $G$-equivariant torsor under the intermediate Jacobian $J^3_{X/k}$ of $X$ (see \Cref{bw-4.5}, which is a $G$-equivariant analogue of \cite[Theorem 4.5]{benoist2023intermediate}). When $X$ is projectively $G$-linear, this implies that $F$ is trivial as a $G$-equivariant $J^3_{X/k}$-torsor, and hence that $F(k)^G\neq\emptyset$, i.e., that $X$ contains a $G$-invariant line.

	\subsection*{Acknowledgments}
	This note was written in July 2021, while the author was a PhD student visiting Universit\'e Paris-Saclay in Orsay. It was inspired by a lecture of Brendan Hassett delivered at the online conference held on the occasion of Zinovy Reichstein's 60th birthday, where \Cref{hassett-tschinkel} was stated and the proof given in \cite{hassett2022equivariant} was sketched. A closely related proof has subsequently appeared in a paper of Ciurca, Tanimoto and Tschinkel \cite{ciurca2024intermediate}. 
    
    I thank Olivier Wittenberg for useful discussions on the subject and for helpful comments on this text. I also thank Ivan Cheltsov and Yuri Tschinkel for their interest in this note and for encouraging me to publish it. I thank the anonymous referee for a careful reading of the manuscript.
	
	\section{Preliminaries}

	\subsection{Equivariant torsors}
	
	Let $G$ be a finite group, and let $A$ be a $G$-module, that is, an abelian group with a left $G$-action. We use multiplicative notation for the group operations of $G$ and $A$, and we denote the $G$-action on $A$ by $(g,a)\mapsto g(a)$. We write $H^1(G,A)$ for the group of $1$-cocycles, i.e., functions $a\colon G\to A$ such that $a_{gh}=a_gg(a_h)$, modulo $1$-coboundaries, that is, functions $G\to A$ of the form $g\mapsto g(a)a^{-1}$ for some fixed $a\in A$. 
	
	A $G$-equivariant $A$-torsor is a non-empty $G$-set $P$ together with a $G$-equivariant left $A$-action (i.e. $g(ax)=g(a)g(x)$ for all $g\in G$, $a\in A$ and $x\in P$) such that the  map $A\times P\to P\times P$ given by $(a,x)\mapsto (ax,x)$ is a bijection. We use round brackets to denote $G$-actions, and juxtaposition to denote $A$-actions.
	
	If $P$ is a $G$-equivariant $A$-torsor, fix $x\in P$ and define the function $a\colon G\to A$ so that $g(x)=a_gx$ for all $g\in G$. Then $a$ is a $1$-cocycle, and defines a class $[P]\in H^1(G,A)$ which does not depend on the choice of $x$. 
	
	Given a $1$-cocycle $a\colon G\to A$, take $P\coloneqq A$ as a left $A$-module, and let $G$ act on $P$ by $g[x]\coloneqq a_gg(x)$. This defines an action, because for all $g,h\in G$ and all $x\in P$ we have
	\[g[h[x]]=a_gg(a_hh(x))=a_gg(a_h)g(h(x))=a_{gh}g(h(x))=(gh)[x].\]
	One checks that equivalent $1$-cocycles give rise to isomorphic $G$-equivariant $A$-torsors, and that the two constructions give rise to bijections \begin{equation}\label{h1-torsors}
	    H^1(G,A)\simeq \set{\text{Isom. classes of $G$-equiv. $A$-torsors}}
	\end{equation}
	that are inverse to each other. The class of the trivial cocycle (i.e. $a_g=1_A$ for all $g\in G$) corresponds to the isomorphism class of the trivial $G$-equivariant $A$-torsor, that is, the left $G$-module $A$, viewed as an $A$-torsor via multiplication on the left.
	
	\begin{lemma}\label{trivial}
		Let $P$ be a $G$-equivariant $A$-torsor. Then $P$ is trivial if and only if $P^G\neq \emptyset$.
	\end{lemma} 
	
	\begin{proof}
		Since $G$ acts on $A$ via group automorphisms, $A^G$ contains $1_A$, thus if $P$ is trivial then $P^G\neq \emptyset$. If $x_0\in P^G$, define a map from the trivial torsor $A$ to $P$ by sending $a\mapsto ax_0$. One easily checks that this is a $G$-equivariant map of $A$-torsors, and that it is bijective.
	\end{proof}
	
	\begin{rmk}
		Let $\mc{T}$ be a topos, and let $A$ be an abelian group object in $\mc{T}$. Then there is a notion of $A$-torsor, and $A$-torsors are classified by the sheaf cohomology group $H^1(\mc{T},A)$. This recovers the previous definition by letting $\mc{T}$ be the topos of $G$-sets, that is, the category of $G$-sets where coverings are jointly surjective maps of $G$-sets. \Cref{trivial} then follows from the fact that an $A$-torsor is trivial if and only if it has a global section.
	\end{rmk}

    Suppose given a short exact sequence of $G$-modules
    \begin{equation}\label{a-b-c}
        0 \to A \xrightarrow{\iota} B \xrightarrow{\pi} C \to 0.
    \end{equation}
    Then $A$ acts on $B$ by translation. For every $c\in C$, the fiber $\pi^{-1}(c)$ is an $A$-torsor; if $c\in C^G$, the fiber $\pi^{-1}(c)$ is a $G$-equivariant $A$-torsor. We have an exact sequence 
    \[0 \to A^G\xrightarrow{\iota} B^G\xrightarrow{\pi} C^G\xrightarrow{\partial}H^1(G,A)\]
    where the connecting homomorphism $\partial$ is defined as follows: for every $c\in C^G$, let $b\in \pi^{-1}(c)$, and for every $g\in G$ let $a_g\in A$ be the unique element such that $\iota(a_g)=g(b)b^{-1}$. For all $g,h\in G$, we have
    \[\iota(a_{gh})=(gh)(b)b^{-1}=g(h(b)b^{-1})g(b)b^{-1}=\iota\bigl(g(a_h)\bigr)\iota(a_g),\]
and hence, as $\iota$ is injective, $a_{gh}=g(a_h)a_g=a_gg(a_h)$. It follows that the function $G\to A$ given by $g\mapsto a_g$ is a $1$-cocycle. By definition, $\partial(c)$ is the class of this cocycle in $H^1(G,A)$.

\begin{lemma}\label{fiber=cocycle}
    Suppose given a short exact sequence of $G$-modules \eqref{a-b-c}, and let $\partial\colon C^G\to H^1(G,A)$ be the corresponding connecting map. For every $c\in C^G$, the class of the $G$-equivariant $A$-torsor $\pi^{-1}(c)$ under the bijection of \eqref{h1-torsors} is equal to $\partial(c)$.
\end{lemma}

\begin{proof}
    Let $c\in C^G$. Recall that we view $\pi^{-1}(c)$ as a $G$-equivariant $A$-torsor via the translation action of $A$. Fix $b\in \pi^{-1}(c)$. By definition, the $1$-cocycle associated with the $G$-equivariant $A$-torsor $\pi^{-1}(c)$ and the base point $b\in \pi^{-1}(c)$ is the unique function $a\colon G\to A$ such that $g(b)=\iota(a_g)b$ for every $g\in G$. This is equivalent to $\iota(a_g)=g(b)b^{-1}$ for every $g\in G$. Thus $a\colon G\to A$ is precisely the $1$-cocycle used in the definition of the connecting map $\partial(c)$. Therefore the class of the torsor $\pi^{-1}(c)$ in $H^1(G,A)$ is exactly $\partial(c)$.
\end{proof}

	\subsection{Jacobians of curves}
	We refer the reader to \cite[\S 2.1]{benoist2020clemens} for the basic definitions on principally polarized abelian varieties.
	
	Let $k$ be an algebraically closed field, and let $C$ be a smooth projective curve over $k$. We denote by $\on{Pic}_{C/k}$ the Picard scheme of $C$ and by $J_{C/k}=\on{Pic}_{C/k}^0$ the Jacobian of $C$. We view $J_{C/k}$ as a principally polarized abelian variety, with polarization given by the class of the theta divisor in the N\'eron-Severi group $\on{NS}(J_{C/k})$. We also let $\on{Pic}(C)\coloneqq \on{Pic}_{C/k}(k)$ and $J(C)\coloneqq J_{C/k}(k)$. We have a short exact sequence of commutative group schemes
	\begin{equation}\label{jac-pic}0\to J_{C/k}\to \on{Pic}_{C/k}\to \on{NS}_{C/k}\to 0.
	\end{equation}
	Let $G$ be a finite group. If $C$ is a $G$-curve, then $G$ also acts on $\on{Pic}_{C/k}$ via group scheme automorphisms, and the induced $G$-action on $J_{C/k}$ respects the theta divisor, hence the polarization.
	
	Assume now that the smooth projective curve $C$ is connected. Then $\on{NS}_{C/k}=\Z$, and (\ref{jac-pic}) reduces to
	\[0\to J_{C/k}\to \on{Pic}_{C/k}\xrightarrow{\on{deg}}\Z \to 0.\]
	For every $d\in \Z$, we let $\on{Pic}_{C/k}^d$ be the connected component of $\on{Pic}_{C/k}$ which parametrizes classes of divisors of degree $d$, and we set $\on{Pic}^d(C)\coloneqq \on{Pic}^d_{C/k}(k)$. If $C$ is a connected $G$-curve, passing to $k$-points and then taking $G$-invariants yields a connecting homomorphism \[\Z\to H^1(G,J(C)),\] which by \Cref{fiber=cocycle} sends $d\in \Z$ to $[\on{Pic}^d(C)]=d[\on{Pic}^1(C)]$. 
	
	\begin{lemma}\label{dec-jac}
		Let $k$ be an algebraically closed field, let $G$ be a finite group, let $C$ be a smooth projective $G$-curve over $k$, let $A\subset J_{C/k}$ be a principally polarized abelian subvariety, and assume that $A$ is $G$-invariant. Then $A=J_{C'/k}$, where $C'\subset C$ is a (possibly empty) $G$-invariant union of connected components of $C$ of positive genus.
	\end{lemma} 
	
	\begin{proof}
		Let $C_1,\dots,C_n\subset C$ be the connected components of $C$ of positive genus. The inclusions $C_i\subset C$ induce embeddings of principally polarized abelian varieties $J_{C_i/k}\subset J_{C/k}$. The $J_{C_i/k}$ are all the indecomposable  factors of $J_{C/k}$, and so there exists $\mc{O}\subset\set{1,\dots,n}$ such that $A=\prod_{i\in \mc{O}}J_{C_i/k}$ as principally polarized abelian subvarieties of $J_{C/k}$. 
		
		The group $G$ acts on $\set{1,\dots,n}$ in the following way: for every $g\in G$ and $i,j\in\set{1,\dots,n}$, we set $g(i)=j$ if $g(C_i)=C_j$. Since $A$ is $G$-invariant, $\mc{O}$ is $G$-invariant. Let $C'\subset C$ be the union of the $C_i$ for $i\in \mc{O}$. Then $A=J_{C'/k}$.
	\end{proof}

	\subsection{Intermediate Jacobians}\label{subsec:i-j}
	Let $k$ be an algebraically closed field, and let $X$ be a rational smooth projective threefold over $k$. We denote by $J^3_{X/k}$ the intermediate Jacobian of $X$, and we let $J^3(X)\coloneqq J^3_{X/k}(k)$. We view $J^3_{X/k}$ as a principally polarized abelian variety, with polarization given by \cite[Property 2.4, Corollary 2.8]{benoist2020clemens}; see also \cite[\S 1.6]{benoist2019intermediate}. We have a short exact sequence of abelian groups
	\begin{equation}\label{intjac-ch2}0\to J^3(X)\to \on{CH}^2(X)\to \on{NS}^2(X)\to 0,
	\end{equation}
	where $\on{NS}^2(X)$ is a free $\Z$-module of finite rank. In particular, $\on{CH}^2(X)$ is represented by a smooth group scheme $\on{CH}^2_{X/k}$ over $k$. Thus we have a short exact sequence of commutative group schemes
	\begin{equation}\label{intjac-ch2'}0\to J^3_{X/k}\to \on{CH}^2_{X/k}\to NS_{X/k}^2\to 0,
	\end{equation}
	where $NS_{X/k}^2$ is the constant group scheme associated to $\on{NS}^2(X)$. 
    
    Let $G$ be a finite group acting on $X$ over $k$. Then $G$ naturally acts on $\on{CH}^2_{X/k}$ via group scheme automorphisms and the sequences (\ref{intjac-ch2}) and (\ref{intjac-ch2'}) become $G$-equivariant. Since cup-product on $H^*(X,\Z_{\ell})$ commutes with pullbacks, it is $G$-invariant. Since the polarization of $J^3_{X/k}$ given in \cite[Property 2.4]{benoist2020clemens} is uniquely determined by the cup-product, it is also $G$-invariant. Thus $G$ acts on $J^3_{X/k}$ via automorphisms of polarized abelian varieties.
	
	Passing to group cohomology in (\ref{intjac-ch2}) yields a connecting homomorphism \[\on{NS}^2(X)^G\to H^1(G,J^3(X)),\] which by \Cref{fiber=cocycle} sends $\alpha\in \on{NS}^2(X)^G$ to $[(\on{CH}^2(X))^{\alpha}]$, the isomorphism class of the coset $\alpha+J^3(X)$.

	\section{A criterion for projective linearity}    
	Let $k$ be an algebraically closed field of characteristic zero, and let $G$ be a finite group. Recall that a $G$-variety $X$ is said to be \emph{projectively $G$-linear} if there exist a group homomorphism $\rho\colon G\to \on{PGL}_{n+1}$ and a $G$-equivariant birational equivalence $X\simeq \P^n_k$, such that $G$ acts on $\P^n_k$ via $\rho$. 
    
    Let $X$ be a projectively $G$-linear smooth projective threefold over $k$. In particular, $X$ is rational, and so it admits an intermediate Jacobian $J^3_{X/k}$. Since $\on{char}(k)=0$, by functorial resolution of indeterminacies there exists a diagram of smooth projective $G$-varieties
	\begin{equation}\label{abhyankar}
		X\xleftarrow{h} X'=Y_{n+1}\to\cdots \to Y_{j+1}\xrightarrow{p_j}Y_j\to\cdots \to Y_1=Y=\P^3_k	
	\end{equation}
	such that $p_j$ is the blow-up of a smooth $G$-invariant closed subscheme $Z_j\subset Y_j$ of codimension $c_j$ and $h$ is projective, birational and $G$-equivariant. Let
	\begin{equation}\label{ch2-decomp}
		\on{CH}^2_{X/k}\times \mc{G}\simeq \on{CH}^2_{\P^3_k/k}\times \prod_{c_j=2}\on{Pic}_{Z_j/k}\times \prod_{c_j=3} \Z_{Z_j/k}
	\end{equation}
	be the isomorphism of group schemes over $k$ constructed in \cite[(3.9)]{benoist2019intermediate} starting from (\ref{abhyankar}). Inspection of the construction of \cite[(3.9)]{benoist2019intermediate} shows that $G$ naturally acts on the group scheme $\mc{G}$ through group automorphisms and that (\ref{ch2-decomp}) is $G$-equivariant. 
	
	\begin{prop}\label{bw-3.1}
		Let $X$ be a projectively $G$-linear smooth projective threefold over an algebraically closed field $k$ of characteristic zero. There exist a smooth projective $G$-curve $B$ over $k$ and a $G$-equivariant embedding $\on{CH}^2_{X/k}\hookrightarrow \on{Pic}_{B/k}$ which identifies $\on{CH}^2_{X/k}$ with a $G$-invariant direct factor of $\on{Pic}_{B/k}$ and such that the induced embedding $J^3_{X/k}\hookrightarrow J_{B/k}$ is a morphism of principally polarized abelian varieties.
	\end{prop}
	
	\begin{proof}
		Let $B$ be the disjoint union of $\P^1_k$, the curves $\P^1_{\pi_0(Z_j/k)}$ for all $j$ such that $c_j=3$, and the curves $Z_j$ for all $j$ such that $c_j=2$. Here $\pi_0(Z_j/k)$ denotes the scheme of connected components of $Z_j$; see \cite[\S 1.3]{benoist2019intermediate}. Then $B$ is a smooth projective curve over $k$, and we have a $G$-equivariant isomorphism \[\on{Pic}_{B/k}\simeq \Z\times\prod_{c_j=2} \on{Pic}_{Z_j/k}\times \prod_{c_j=3}\Z_{Z_j/k}.\]
		The conclusion follows from (\ref{ch2-decomp}).
	\end{proof}

	\begin{prop}\label{bw-3.10}
		Let $X$ be a projectively $G$-linear smooth projective threefold over an algebraically closed field $k$ of characteristic zero. For every smooth projective connected $G$-curve $D$ of genus $\geq 2$ over $k$, every $G$-equivariant morphism $\psi\colon J^3_{X/k}\to J_{D/k}$ identifying $J_{D/k}$ with a principally polarized factor of $J^3_{X/k}$, and every $\alpha\in \on{NS}^2(X)^G$, there exists $d\in \Z$ such that $\psi_*[(\on{CH}^2(X))^{\alpha}]=[\on{Pic}^d(D)]$ in $H^1(G,J(D))$.	
	\end{prop}
	
	\begin{proof}
		Let $B$ be the $G$-curve given by  \Cref{bw-3.1}, let $q\colon \on{Pic}_{B/k}\to \on{CH}^2_{X/k}$ be the induced projection map, and let $q^0$ be the restriction of $q$ to $J_{B/k}$. The composite $\psi\circ q^0\colon J_{B/k}\to J_{D/k}$ is $G$-equivariant and exhibits $J_{D/k}$ as a $G$-invariant principally polarized direct factor of $J_{B/k}$. By \Cref{dec-jac}, there exist a $G$-invariant union $B'$ of connected components of $B$ of positive genus and a $G$-equivariant isomorphism $J_{B'/k}\simeq J_{D/k}$. Since $D$ is connected of genus $\geq 2$, its Jacobian $J_{D/k}$ is indecomposable of dimension $\geq 2$, hence so is $J_{B'/k}$, and therefore $B'$ is connected of genus $g\geq 2$. By the precise form of Torelli's Theorem \cite[Appendice, Th\'eor\`emes 1 et 2]{lauter2001geometric}, after possibly replacing $q$ by $-q$ if $D$ is not hyperelliptic (of course, if $q$ is a $G$-equivariant split surjection, so is $-q$), we may identify $B'$ and $D$ so that $\psi\circ q^0$ is the pullback of the composition $i\colon D\simeq B'\subset B$. Since $q$ realizes $\on{CH}^2_{X/k}$ as a $G$-invariant direct factor of $\on{Pic}_{B/k}$, the induced map $\cl{q}\colon \on{NS}(B)\to \on{NS}^2(X)$ is surjective and $G$-equivariantly split, and so induces a surjective map $\on{NS}(B)^G\to (\on{NS}^2(X))^G$. Let $\beta \in \on{NS}(B)^G$ be a lifting of $\alpha$, and let $d\coloneqq i^*\beta\in \on{NS}(D)\simeq \Z$. We have the following commutative diagram of $G$-equivariant maps and exact rows:
		\[
		\begin{tikzcd}
			0 \arrow[r] & J^3(X) \arrow[r] & \on{CH}^2(X) \arrow[r] & \on{NS}^2(X)\arrow[r] & 0 \\
			0 \arrow[r] & J(B) \arrow[r] \arrow[u, "q^0"] \arrow[d,"i^*"]& \on{Pic}(B) \arrow[r] \arrow[u,"q"] \arrow[d, "i^*"]& \on{NS}(B) \arrow[r] \arrow[u,"\cl{q}"] \arrow[d, "i^*"]& 0 \\
			0 \arrow[r] & J(D) \arrow[r]& \on{Pic}(D) \arrow[r]& \on{NS}(D) \arrow[r]& 0.   	
		\end{tikzcd}
		\]
		Since we have identified $\psi\circ q^0$ with the pullback of $i$, passing to group cohomology we obtain
		\[\psi_*[(\on{CH}^2(X))^{\alpha}]=(\psi\circ q^0)_*[\on{Pic}(B)^{\beta}]=[\on{Pic}^d(D)]\in H^1(G, J(D)).\qedhere\] 
	\end{proof}

	\section{Smooth complete intersections of two quadrics}

	\begin{prop}\label{bw-4.5}
		Let $X\subset \P^5_k$ be a smooth complete intersection of two quadrics, let $F$ be its variety of lines (i.e. the Hilbert scheme of lines in $\mathbb{P}^5_k$ which are contained in $X$), and let $Z\subset X\times F$ be the universal line. Let $G$ be a finite subgroup of $\on{PGL}_6$ such that $X$ is stable under the $G$-action. 
		\begin{enumerate}[label=(\roman*)]
			\item We have a $G$-equivariant short exact sequence of group schemes
			\[0\to J^3_{X/k}\to \on{CH}^2_{X/k}\xrightarrow{\delta}\Z \to 0,\]
			where $\delta$ is induced by the degree map $\deg \colon \on{CH}^2(X)\to \Z$. For every $n\in\Z$, we let $(\on{CH}^2_{X/k})^n\coloneqq \delta^{-1}(n)$. 
			\item We have a $G$-equivariant isomorphism $F\simeq (\on{CH}^2_{X/k})^1$.
			\item There exists a unique $G$-invariant reduced closed $k$-subscheme $D\subset \on{CH}^2_{X/k}$ such that $D(k)\subset \on{CH}^2(X)$ is the subset of classes of conics contained in $X$. The scheme $D$ is a smooth projective connected curve of genus $2$ over $k$.
			\item The inclusion $D\subset \on{CH}^2_{X/k}$ induces a $G$-equivariant isomorphism of principally polarized abelian varieties \[J_{D/k}\simeq J^3_{X/k}\] and, after identifying $J_{D/k}$ and $J^3_{X/k}$ via this isomorphism, a $G$-equivariant isomorphism of $G$-equivariant torsors \[\on{Pic}^1_{D/k}\simeq (\on{CH}^2_{X/k})^2.\]
		\end{enumerate}	
	\end{prop}
	Taking $k$-points and passing to group cohomology in (i) gives a connecting homomorphism $\Z\to H^1(G,J^3_{X/k}(k))$ which sends $n\in \Z$ to $[(\on{CH}^2_{X/k})^n(k)]$; see \S\ref{subsec:i-j}.
    
	\begin{proof}
		(i) This is \cite[Theorem 4.5(i)]{benoist2019intermediate}.
		
		(ii) By the universal property of the Hilbert scheme, the group $G$ acts on $F$, and the diagonal $G$-action on $X\times F$ leaves $Z$ invariant. By \cite[Theorem 4.5(ii)]{benoist2019intermediate}, we have an isomorphism $F\simeq (\on{CH}^2_{X/k})^1$ which is induced by $[\mc{O}_Z]\in K_0(X_F)$ and the universal property of $\on{CH}^2_{X/k}$. Since $Z$ is $G$-invariant, this isomorphism is $G$-equivariant.
		
		(iii) Existence and uniqueness of $D$, as well as the fact that it is a smooth projective curve of genus $2$, are proved in  \cite[Theorem 4.5(iii)]{benoist2019intermediate}. The $G$-invariance immediately follows from the uniqueness and the fact that $G$ acts on $\P^5_k$ via $\on{PGL}_6$, and so the $G$-action sends conics to conics.
		
		(iv) As $D$ is $G$-invariant, the isomorphisms $J_{D/k}\simeq J^3_{X/k}$ and $\on{Pic}^1_{D/k}\simeq (\on{CH}^2_{X/k})^2$ defined in \cite[Theorem 4.5(iv)]{benoist2019intermediate} are $G$-equivariant.
	\end{proof}
	
	\begin{proof}[Proof of \Cref{hassett-tschinkel}]
		Assume that $X$ contains a $G$-invariant line $L$. If $V$ is the variety of planes of $\P^5_k$ containing $L$, then projecting from $L$ gives a $G$-equivariant birational morphism $X\setminus L\to V$. Since $V$ is isomorphic to $\P^3_k$, we conclude that $X$ is projectively $G$-linear, as desired. 	
		
		Assume now that $X$ is projectively $G$-linear. Let $D$ be the $G$-curve of \Cref{bw-4.5}(iii), and let $\psi\colon  J^3_{X/k}\xrightarrow{\sim} J_{D/k}$ be the inverse of the isomorphism of \Cref{bw-4.5}(iv). By \Cref{bw-4.5}(i), we have $\on{NS}^2(X)=\Z$ with trivial $G$-action. By \Cref{bw-3.10} (applied to $\alpha=1$) we have \[\psi_*[(\on{CH}^2_{X/k})^1(k)]=[\on{Pic}^d(D)]\in H^1(G,J(D)).\] By \Cref{bw-4.5}(iv), \[[\on{Pic}^1(D)]=2\psi_*[(\on{CH}^2_{X/k})^1(k)]\in H^1(G,J(D)).\] Combining these two equalities, we obtain 
		\begin{equation}\label{combine}
			\psi_*[(\on{CH}^2_{X/k})^1(k)]=[\on{Pic}^d(D)]=[\on{Pic}^{1-d}(D)]\in H^1(G,J(D)).
		\end{equation}
		Since $D$ has genus $2$, the canonical class $K_D$ belongs to $\on{Pic}^2(D)^G$, hence by \Cref{trivial} the $G$-equivariant $J(D)$-torsor $\on{Pic}^2(D)$ is trivial. As one of $d$ and $1-d$ is even, it follows from (\ref{combine}) that $\psi_*[(\on{CH}^2_{X/k})^1(k)]=0$ in $H^1(G,J(D))$. Since $\psi$ is an isomorphism, we deduce that the $G$-equivariant $J^3(X)$-torsor $(\on{CH}^2_{X/k})^1(k)$ is trivial. By \Cref{trivial}, this implies that $(\on{CH}^2_{X/k})^1(k)$ has a $G$-invariant element. By \Cref{bw-4.5}(ii), we have a $G$-equivariant bijection $(\on{CH}^2_{X/k})^1(k)\simeq F(k)$, and hence $F(k)^G\neq \emptyset$, that is, $X$ contains a $G$-invariant line.  
	\end{proof}

    \begin{rmk}
    The following remark complements \Cref{bw-4.5}(ii). The $G$-variety $F$ is naturally a $G$-equivariant torsor under its Albanese variety $\on{Alb}^0_{F/k}$. By \cite[Theorem 4.5(ii)]{benoist2019intermediate}, the class $[\mc{O}_Z]\in K_0(X_F)$ induces a $G$-equivariant isomorphism $\on{Alb}^0_{F/k}\simeq J^3_{X/k}$, compatible with the $G$-equivariant isomorphism $F\simeq (\on{CH}^2_{X/k})^1$ of \Cref{bw-4.5}(ii). Equivalently, we have a commutative square of $G$-equivariant maps
    \[
    \begin{tikzcd}
        \on{Alb}^0_{F/k}\times F \arrow[r] \arrow[d, "\wr"] 
        & F \arrow[d, "\wr"] \\
        J^3_{X/k}\times (\on{CH}^2_{X/k})^1 \arrow[r] 
        & (\on{CH}^2_{X/k})^1,
    \end{tikzcd}
    \]
    where the horizontal maps are given by the torsor actions.
\end{rmk}

\end{document}